\documentclass{amsart}
\usepackage{amsfonts,amssymb,amscd}

\input{xy}
\xyoption{all}

\newcommand{\IR}{\mathbb R}
\newcommand{\IN}{\mathbb N}
\newcommand{\II}{\mathbb I}
\newcommand{\w}{\omega}
\newcommand{\cbox}{\operatornamewithlimits{\boxdot}}

\newcommand{\U}{\mathcal U}

\newcommand{\id}{\mathrm{id}}

\newcommand{\ulim}{\mathrm u\mbox{-}\kern-2pt\varinjlim}
\newcommand{\tlim}{\mathrm t\mbox{-}\kern-2pt\varinjlim}
\newcommand{\lclim}{\mathrm{lc}\mbox{-}\kern-2pt\varinjlim}
\newcommand{\llim}{\mathrm{l}\mbox{-}\kern-2pt\varinjlim}

\newcommand{\glim}{\mathrm{g}\mbox{-}\kern-2pt\varinjlim}
\newcommand{\LL}{{\mathsf L}}
\newcommand{\RR}{{\mathsf R}}
\newcommand{\LR}{\mathsf{LR}}
\newcommand{\RL}{\mathsf{RL}}

\newcommand{\lz}{lz}

\newtheorem{theorem}{Theorem}[section]
\newtheorem{lemma}[theorem]{Lemma}
\newtheorem{problem}[theorem]{Problem}
\newtheorem{corollary}[theorem]{Corollary}
\newtheorem{proposition}[theorem]{Proposition}

\newtheorem{claim}[theorem]{Claim}
\theoremstyle{definition}
\newtheorem{definition}[theorem]{Definition}
\newtheorem{remark}[theorem]{Remark}

\title[On topological groups homeomorphic to LF-spaces]{Detecting topological groups which are\\ (locally) homeomorphic to LF-spaces}
\thanks{The first and third authors were partly supported by Slovenian Research Agency grant P1-0292-0101, J1-4144-0101, and J1-9643-0101. The last author was supported by JSPS Grant-in-Aid for Scientific Research (No.22540081).}
\author{T.~Banakh}
\address[T.Banakh]{Institute for Mathematics, Physics and Mechanics,
University of Ljubljana (Slovenia), Jan Kochanowski University in Kielce (Poland) and Ivan Franko national University of Lviv (Ukraine)}
\email{tbanakh@yahoo.com, T.O.Banakh@gmail.com}

\author{K.~Mine}
\address[K.Mine]{Institute of Mathematics,
University of Tsukuba, Tsukuba, 305-8571, Japan}
\email{pen@math.tsukuba.ac.jp}

\author{D.~Repov\v s}
\address[D.Repov\v s]{Faculty of Education, and Faculty of Mathematics and Physics,
University of Ljubljana,
P. O. Box 2964,
Ljubljana, Slovenia 1001}
\email{dusan.repovs@guest.arnes.si}

\author{K.~Sakai}
\address[K.Sakai]{Institute of Mathematics,
University of Tsukuba, Tsukuba, 305-8571, Japan}
\email{sakaiktr@sakura.cc.tsukuba.ac.jp}

\author{T.~Yagasaki}
\address[T.Yagasaki]{Division of Mathematics,
Graduate School of Science and Technology,
Kyoto Institute of Technology, Kyoto, 606-8585, Japan}
\email{yagasaki@kit.ac.jp}

\begin{document}
\begin{abstract} We prove that a topological group $G$ is (locally) homeomorphic to an LF-space if $G=\bigcup_{n\in\w}G_n$ for some increasing sequence of subgroups $(G_n)_{n\in\w}$ such that
\begin{enumerate}
\item for any  neighborhoods $U_n\subset G_n$, $n\in\w$, of the neutral element $e\in G_n\subset G$,  the set $\bigcup\limits_{n=1}^\infty U_0U_1\cdots U_n$ is a neighborhood of $e$ in $G$;
\item each group $G_n$ is (locally) homeomorphic to  a Hilbert space;
\item for every $n\in\IN$ the quotient map $G_n\to G_n/G_{n-1}$ is a locally trivial bundle;
\item for infinitely many numbers $n\in\IN$ each $Z$-point in the quotient space $G_{n}/G_{n-1}=\{xG_{n-1}:x\in G_n\}$ is a strong $Z$-point.
\end{enumerate}
\end{abstract}

 \maketitle

\section{Introduction}
The problem of recognizing the topological structure of topological groups traces its history back to the fifth problem of Hilbert which asks if Lie groups can be characterized as topological groups whose underlying topological spaces are manifolds. This problem was resolved by combined efforts of Gleason \cite{Gle}, Montgomery and Zippin
\cite{MZ}, Hoffman \cite{Hof}, and Iwasawa \cite{Iwa} who proved the following

\begin{theorem} A topological group $G$ is \textup{(}locally\textup{)} homeomorphic to a finite-dimensional Hilbert space if and only if $G$ is locally compact and \textup{(}locally\textup{)} contractible.
\end{theorem}

We say that a topological space $X$ is {\em locally homeomorphic} to a space $E$ if each point $x\in X$ has an open neighborhood homeomorphic to an open subset of $E$. If in addition, $X$ is paracompact, then $X$ is called an {\em $E$-manifold}. A {\em Hilbert manifold} is a
paracompact space, locally homeomorphic to a Hilbert space.

Topological groups (locally) homeomorphic to separable Hilbert spaces were characterized by Dobrowolski and Toru\'nczyk \cite{DT}:

\begin{theorem}[Dobrowolski-Toru\'nczyk]\label{DT} A topological group $G$ is \textup{(}locally\textup{)} homeomorphic to a separable Hilbert space if and only if $G$ is a \textup{(}locally\textup{)} Polish absolute \textup{(}neighborhood\textup{)} retract.
\end{theorem}

In this theorem a Hilbert space can be finite or infinite dimensional. A topological space is called {\em locally Polish} if each point has a Polish (i.e., separable completely metrizable) neighborhood.

Topological groups which are (locally) homeomorphic to non-separable Hilbert spaces were characterized by Banakh and Zarichnyi \cite{BZar}:

\begin{theorem}[Banakh-Zarichnyi] A topological group $G$ is \textup{(}locally\textup{)} homeomorphic to an infinite-dimensional Hilbert space if and only if $G$ is a completely metrizable absolute (neighborhood) retract  which satisfies LFAP.
\end{theorem}

We say that a topological space $X$ satisfies LFAP (the {\em Locally Finite Approximation Property\/}) if for any open cover $\U$ of $X$ there is a sequence of maps $f_n:X\to X$, $n\in\w$, such that each map $f_n$ is $\U$-near to the identity and the family $\big(f_n(X)\big)_{n\in\w}$ is locally finite in $X$. Two maps $f,g:X\to Y$ are called {\em $\U$-near} for a cover $\U$ of $Y$ if for each point $x\in X$ the doubleton $\{f(x),g(x)\}$ lies in some set $U\in\U$.
\smallskip

In this paper we address the problem of recognizing topological groups which are (locally) homeomorphic to LF-spaces.

We recall that an {\em LF-space} is the direct limit $\lclim X_n$ of a tower
$$X_0\subset X_1\subset X_2\subset\cdots
$$of Fr\'echet (i.e., locally convex linear completely metrizable) spaces in the category of locally convex spaces. More precisely, $\lclim X_n$ is the union $X=\bigcup_{n\in\w}X_n$ endowed with the strongest topology that turns $X$ into a locally convex space and makes the identity inclusions $X_n\to X$, $n\in\w$, continuous.

The simplest non-trivial example of an LF-space is $\IR^\infty$, the direct limit of the tower
$$\IR^1\subset \IR^2\subset \IR^3\subset\cdots,$$
where each space $\IR^n$ is identified with the hyperplane $\IR^n\times\{0\}$ in $\IR^{n+1}$. The space $\IR^\infty$ can be identified with the direct sum $\bigoplus_{n\in\w}\IR$ of one-dimensional Banach  spaces in the category of locally convex spaces.

The topological classification of LF-spaces was obtained by Mankiewicz \cite{Man}  who proved that each LF-space is homeomorphic to the direct sum $\bigoplus_{n\in\w}l_2(\kappa_i)$ of Hilbert spaces for some sequence  of cardinals $(\kappa_i)_{i\in\w}$. Here $l_2(\kappa)$ denotes the Hilbert space with an orthonormal base of cardinality $\kappa$. A more precise version of Mankiewicz's classification says that the following spaces
\begin{itemize}
\item $l_2(\kappa)$ for some cardinal $\kappa\ge 0$,
\item $\IR^\infty$,
\item $l_2(\kappa)\times\IR^\infty$ for some $\kappa\ge\w$, and
\item $\bigoplus_{n\in\w}l_2(\kappa_i)$ for a strictly increasing sequence of infinite cardinals $(\kappa_i)_{i\in\w}$
\end{itemize}
are pairwise non-homeomorphic and represent all possible topological types of LF-spaces. In particular, each infinite-dimensional separable LF-space is homeomorphic to one of the following spaces: $l_2$, $\IR^\infty$ or $l_2\times\IR^\infty$. The topological characterizations of the LF-spaces $l_2$ and $\IR^\infty$ were given by Toru\'nczyk \cite{Tor81}, \cite{Tor85} and Sakai \cite{Sak84}, respectively. Other LF-spaces were recently characterized by Banakh and Repov\v s \cite{BR-LF}.

The description of the topology of the direct sum $\bigoplus_{n\in\w}X_n$ of locally convex spaces given in \cite[II.\S6.1]{Sh} implies that this topology  coincides with the topology of the small box-product $\cbox_{n\in\w}X_n$.
The construction of the small box-product $\cbox_{n\in\w}X_n$ of pointed topological spaces is purely topological and is defined as follows.

By a {\em pointed space} $X$ we understand a space with a distinguished point, which will be denoted by $*_X$.  Each group $G$ is a pointed space whose distinguished point $*_G$ is the neutral element of $G$. For a subgroup $H\subset G$ the quotient space $G/H=\{xH:x\in G\}$ is a pointed space with the distinguished point $*_{G/H}=H\in G/H$.

The {\em small box-product} of a sequence $(X_n)_{n\in\w}$ of pointed topological spaces is the subspace
$$\cbox_{n\in\w}X_n=\{(x_n)_{n\in\w}\in\square_{n\in\w}X_n:\exists m\in\w\;\forall n\ge m\;\;x_n= *_{X_n}\}$$of the box-product $\square_{n\in\w}X_n$. The latter space is the product $\prod_{n\in\w}X_n$ endowed with the topology generated by the products $\prod_{n\in\w}U_n$ of open subsets $U_n\subset X_n$, $n\in\w$.

Now let us return to the problem of recognizing topological groups that are (locally) homeomorphic to LF-spaces.
 We say that a topological group $G$ {\em carries the strong topology with respect to a tower of subgroups} $$G_0\subset G_1\subset G_2\subset\cdots$$if $G=\bigcup_{n\in\w}G_n$ and for any neighborhoods $U_n\subset G_n$, $n\in\w$, of the neutral element $e$ the group product $\bigcup_{n=1}^\infty U_0U_1\cdots U_n$ is a neighborhood of $e$ in the group $G$. The nature of this property will be discussed in Section~\ref{s:stg}.

A closed subgroup $H$ of a topological group $G$ is defined to be ({\em locally}) {\em topologically complemented} in $G$ if the quotient map $\pi:G\to G/H$, $\pi:x\mapsto xH$, is a (locally) trivial bundle. This happens if and only if $\pi$ has a section $s:G/H\to G$, which is continuous on (some non-empty open subset of) the quotient space $G/H$. It follows that for a (locally) topologically complemented subgroup $H$ of $G$ the group $G$ is (locally) homeomorphic to the product $H\times (G/H)$.
In Proposition~\ref{p:ltc-tc} we shall prove that a locally topologically complemented subgroup $H$ of an ANR-group $G$ is topologically complemented in $G$ if the quotient space $G/H$ is contractible or both groups $G$ and $H$ are contractible. By an {\em ANR-group} we understand a topological group whose underlying topological space is an ANR. Therefore, each ANR-group is metrizable.

A tower of groups $(G_n)_{n\in\w}$ is called ({\em locally}) {\em topologically complemented} if each group $G_n$ is (locally) topologically complemented in $G_{n+1}$.

The topology of a topological group carrying the strong topology with respect to a (locally) topologically complemented tower of subgroups is closely related to small box-products:

\begin{theorem}\label{t:gbox} A topological group $G$ carrying the strong topology with respect to a (locally) topologically complemented tower of subgroups $(G_n)_{n\in\w}$ is (locally) homeomorphic to the small box-product $G_0\times \cbox\limits_{n\in\w}G_{n+1}/G_n$.
\end{theorem}

Theorem~\ref{t:gbox} will be proved in Section~\ref{s:gbox}. This theorem motivates the problem of studying the topological structure of small box-products and recognizing small box-products that are (locally) homeomorphic to LF-spaces. A corresponding criterion was proved in \cite{BR-LF}. It involves the notion of a strong $Z$-point.

Let us recall that a closed subset $A$ of a topological space $X$ is called a ({\em strong}) {\em $Z$-set in} $X$ if for any open cover $\U$ of $X$ there is a continuous map $f:X\to X$ such that $f$ is $\U$-near to the identity $\id_X:X\to X$ and (the closure $\overline{f(X)}$ of) the set $f(X)$ does not intersect $A$.
It is clear that each strong $Z$-set is a $Z$-set. The converse is not true, see \cite{BBMW}. However each $Z$-set in a Hilbert manifold is a strong $Z$-set, see \cite{BM}, \cite{Tor85}. A point $x$ of a topological space $X$ will be called a ({\em strong}) {\em $Z$-point} if the singleton $\{x\}$ is a (strong) $Z$-set in $X$.

A pointed topological space $X$ is called {\em $\lz$-pointed} if the distinguished point $*_X$ is not isolated in $X$ and either $X$ is locally compact or $*_X$ is a strong $Z$-point in $X$.

We shall use the following criterion proved in \cite{BR-LF}:

\begin{theorem}[Banakh-Repov\v s]\label{t:sb} The small box-product $\cbox_{n\in\w}X_n$ of pointed topological spaces $X_n$, $n\in\w$, is homeomorphic to (an open subspace of) an LF-space if for every $n\in\w$ the finite product $\prod_{i\le n}X_i$ is homeomorphic to (an open subset of) a Hilbert space and the space $X_n$ is $\lz$-pointed for infinitely many numbers $n\in\w$.
\end{theorem}

 We shall say that a topological space has {\em the $Z$-point property} if each $Z$-point in $X$ is a strong $Z$-point. Since each $Z$-set in a Hilbert manifold is a strong $Z$-set, Theorem~\ref{DT} implies that each Polish ANR-group has the $Z$-point property.

In Proposition~\ref{p:lz} we shall prove that for a locally topologically complemented subgroup $H$ of an ANR-group $G$ the quotient space $G/H$ is $\lz$-pointed if and only if it is not discrete and has the $Z$-point property. Combining this fact with Theorems~\ref{t:gbox}, \ref{t:sb} and Proposition~\ref{p:ltc-tc}, we obtain the following criterion, which is the main result of this paper.
This criterion has been applied in \cite{BMSY} and \cite{BY} for recognizing the topology of some homeomorphism and diffeomorphism groups.

\begin{theorem}\label{t:gLF} A topological group $G$ carrying the strong topology with respect to a tower of subgroups $(G_n)_{n\in\w}$ is
\begin{enumerate}
\item homeomorphic to (an open subset of) an LF-space if for every $n\in\w$ the group $G_n$ is homeomorphic (to an open subset) of a Hilbert space, $G_n$ is topologically complemented in $G_{n+1}$, and for infinitely many numbers $n\in\w$ the quotient space $G_{n+1}/G_n$ is not discrete and has the $Z$-point property;
\item (locally) homeomorphic to an LF-space if for every $n\in\IN$ the  group $G_n$ is (locally) homeomorphic to a Hilbert space, $G_n$ is locally topologically complemented in $G_{n+1}$, and for infinitely many numbers $n\in\w$ the quotient space $G_{n+1}/G_n$ is not discrete and has the $Z$-point property.
\end{enumerate}
\end{theorem}

Because of the lack of an Open Embedding Theorem for LF-manifolds, we distinguish between LF-manifolds and open subspaces of LF-spaces. This is why we have two separate statements (1) and (2) in Theorem~\ref{t:gLF}. It should be mentioned that  the topological structure of open subspaces of LF-spaces is quite well understood, which cannot be said about LF-manifolds, see \cite{MS}, \cite{MS2}.

In light of Theorem~\ref{t:gLF} it is natural to ask if the quotient spaces $G_{n+1}/G_n$ always have the $Z$-point property.

\begin{problem} Let $G$ be a Polish ANR-group and $H$ be a (locally) topologically complemented subgroup in $G$. Does the quotient space $G/H$ have the $Z$-point property?
\end{problem}

The answer to this problem is (trivially) affirmative if $G/H$ is a Hilbert manifold. This is why Theorem~\ref{t:gLF} implies the following criterion for recognition of topological groups which are locally homeomorphic to LF-spaces.

\begin{theorem}\label{t:gGH} A topological group $G$ carrying the strong topology with respect to a tower of subgroups
$$\{*_G\}=G_0\subset G_1\subset \cdots$$
\begin{enumerate}
\item is homeomorphic to (an open subset of) an LF-space if for every $n\in\w$ the group $G_{n}$ is topologically complemented in $G_{n+1}$ and the quotient space $G_{n+1}/G_n$ is homeomorphic to (an open subset of) a Hilbert space;
\item is (locally) homeomorphic to an LF-space if for every $n\in\IN$ the  group $G_n$ is locally topologically complemented in $G_{n+1}$ and the quotient space $G_{n+1}/G_n$ is (locally) homeomorphic to a Hilbert space.
\end{enumerate}
\end{theorem}

Considering Theorem~\ref{t:gGH}, we can ask another open

\begin{problem} Let $G$ be a Polish ANR-group and $H$ be a (locally) topologically complemented subgroup in $G$. Is $G/H$ a Hilbert manifold?
\end{problem}

Banakh and Repov\v s obtained in \cite[2.2]{BR10a} an affirmative answer to this problem under the condition that the subgroup $H$ is {\em balanced} in $G$. The latter means that for each neighborhood $U\subset G$ of the neutral element $*_G$ of $G$ there is a neighborhood $V\subset G$ of $*_G$ such that $VH\subset HU$.

\begin{theorem}[Banakh-Repov\v s] If $G$ is a Polish ANR-group and $H$ is a balanced closed ANR-subgroup in $G$, then the quotient space $G/H$ is a Hilbert manifold and hence has the $Z$-point property.
\end{theorem}

Combining this theorem with Theorems~\ref{DT}, \ref{t:gLF} and Proposition~\ref{p:ltc-tc}, we obtain the following criterion.

\begin{theorem}\label{main3} A topological group $G$ carrying the strong topology with respect to a (locally) topologically complemented tower of Polish ANR-groups $(G_n)_{n\in\w}$
is (locally) homeomorphic to an open subset of the LF-space $\IR^\infty$ or $l_2\times \IR^\infty$ if for infinitely many numbers $n\in\w$ the subgroup $G_n$ is balanced and not open in $G_{n+1}$.
\end{theorem}

Next, we formulate another condition on a subgroup $H$ of a topological group $G$ which implies that the quotient space $G/H$ has the $Z$-point property. Note that the quotient space $G/H$ is a $G$-space with the natural left action of the group $G$.

Let us recall that a {\em $G$-space} is a topological space $X$ endowed with a continuous action $\alpha:G\times X\to X$, $\alpha:(g,x)\mapsto gx$, of a topological group $G$. We say that the action of $G$ on $X$ is {\em locally bounded} at a point $x_0\in X$ if there is a neighborhood $U\subset G$ of the neutral element $*_G$ of $G$ such that for every neighborhood $V\subset X$ of $x_0$ there is a compact subset $K\subset X$ which meets each shift $xV$, $x\in U$. In the opposite case, the action is called {\em locally unbounded} at $x_0$. The action of $G$ on $X$ is {\em locally unbounded} if this action is locally unbounded at each point $x_0\in X$. If the action of $G$ is locally unbounded at some point $x_0\in X$, then $X$ is not locally compact at $x_0$.

In Proposition~\ref{p:Zlb} we shall show that for a closed subgroup $H$ of a locally path-connected topological group $G$, each point of the quotient space $G/H$ is a strong $Z$-point if the space $G/H$ is a separable ANR and the action of $G$ on $G/H$ is locally unbounded. Combining this fact with Theorems \ref{DT} and \ref{t:gLF}, we get the following criterion:

\begin{corollary} A topological group $G$ carrying the strong topology with respect to a (locally) topologically complemented tower of Polish ANR-groups $(G_n)_{n\in\w}$
is (locally) homeomorphic to an open subset of the LF-space $l_2\times \IR^\infty$ if for infinitely many numbers $n\in\w$ the action of the group $G_{n+1}$ on $G_{n+1}/G_n$ is locally unbounded.
\end{corollary}

\section{Uniform direct limits}

It turns out that the structure of topological groups $G$ carrying the strong topology with respect to a tower of subgroups $(G_n)_{n\in\w}$ can be described in terms of uniform direct limits of towers of uniform spaces.
Therefore, in this section we recall the necessary  information on this topics.
For basic information on uniform spaces we refer the reader to Chapter 8 of Engelking's monograph \cite{En}.

All topological spaces considered in this paper are completely regular and all maps are continuous. For a uniform space $X$ we denote its uniformity by $\U_X$. Since uniform spaces are completely regular, the intersection $\cap\U_X$ coincides with the diagonal of $X\times X$. A uniform space $X$ is called {\em metrizable} if its uniformity is generated by some metric. Elements of the uniformity $\U_X$ are called {\em entourages}. For an entourage $U\in\U_X$, a point $x\in X$ and a subset $A\subset X$ by $B(x,U)=\{y\in X:(x,y)\in U\}$ we denote the $U$-ball centered at $x$ and by $B(A,U)=\bigcup_{x\in A}B(x,U)$ the $U$-neighborhood of $A$. A subset $O(A)$ is called a {\em uniform neighborhood} of $A$ in $X$ if $B(A,U)\subset O(A)$ for some entourage $U\in\U_X$.

By a {\em tower} of uniform spaces we shall understand an increasing sequence $$X_0\subset X_1\subset X_2\subset\cdots
$$of uniform spaces (so the uniformity of each space $X_n$ coincides with the uniformity inherited from the uniform space $X_{n+1}$).

For a tower of uniform spaces $$X_0\subset X_1\subset X_2\subset\cdot$$its {\em uniform direct limit} $\ulim X_n$ is the union $X=\bigcup_{n\in\w}X_n$ endowed with the largest uniformity making the identity inclusions $X_n\to X$, $n\in\w$,  uniformly continuous. The topology and the uniformity of the uniform direct limit $\ulim X_n$ were described in \cite{BR10}.

If each space $X_n$ of the tower is locally compact then the topology of $\ulim X_n$ coincides with the topology of the topological direct limit $\tlim X_n$ of the tower $(X_n)_{n\in\w}$. The {\em topological direct limit} $\tlim X_n$ of a tower $(X_n)_{n\in\w}$ of topological spaces is the union $\bigcup_{n\in\w}X_n$ endowed with the largest topology turning the identity inclusions $X_n\to X$, $n\in\w$, into continuous maps.

If $(X_i)_{n\in\w}$ is a sequence of pointed uniform spaces, then each finite (box-)product
$$\cbox_{i\le n}X_i=\{(x_i)_{i\in \w}\in\cbox_{i\in\w}X_i:\forall i>n\;\;x_i=*_{X_i}\}\subset\cbox_{i\in\w}X_i$$carries the product uniformity.
Therefore, $(\cbox_{i\le n}X_i)_{n\in\w}$ turns into a tower of uniform spaces whose union $\bigcup_{n\in\w}\cbox_{i\le n}X_i$ coincides with the small box-product $\cbox_{n\in\w}X_i$. The following lemma was proved in \cite[5.5]{BR10}.

\begin{lemma} For a sequence $(X_i)_{i\in\w}$ of pointed uniform spaces the identity map$$\ulim \cbox_{i\le n}X_i\to\cbox_{i\in\w}X_i$$is a homeomorphism.
\end{lemma}

Next, we recall the definition of a (locally) complemented subset of a uniform space, introduced in \cite{BR-LF}.

\begin{definition} Let $Z$ be a pointed topological space. A subset $A$ of a uniform space $X$ is called {\em $Z$-complemented} in $X$ if there is a homeomorphism $\gamma:A\times Z\to X$ such that
\begin{enumerate}
\item for any neighborhood $V\subset Z$ of $*_Z$ there is an entourage $U\in\U_X$ such that $B(A,U)\subset \gamma(A\times V)$, and
\item for any entourage $U\in\U_X$ there is a neighborhood $V\subset Z$ of $*_Z$ such that  $\gamma(\{a\}\times V)\subset B(a,U)$ for each $a\in A$.
\end{enumerate}
A subset $A\subset X$ is called {\em locally $Z$-complemented} in $X$ if for some open neighborhood $V\subset Z$ of $*_Z$  the set $A$ is $V$-complemented in some open uniform neighborhood $U(A)$ of $A$ in $X$.
\end{definition}

The following important fact was proved in \cite{BR-LF}.

\begin{lemma}\label{l:udlbp} Let $(Z_n)_{n\in\w}$ be a sequence of pointed topological spaces and $(X_n)_{n\in\w}$ be a tower of uniform spaces.
\begin{enumerate}
\item If each set $X_n$ is $Z_n$-complemented in $X_{n+1}$, then the uniform direct limit $\ulim X_n$ is homeomorphic to the small box-product $X_0\times\cbox_{n\in\w}Z_n$.
\item If each set $X_n$ is locally $Z_n$-complemented in $X_{n+1}$, then each point $x_0\in X_0$ has an open neighborhood $O(x_0)\subset \ulim X_n$ which is homeomorphic to an open subset of $X_0\times \cbox_{n\in\w}Z_n$.
\end{enumerate}
\end{lemma}

\section{Strong topology on topological groups}\label{s:stg}

In this section we shall study the  structure of  topological groups $G$ that carry the strong topology with respect to a tower of subgroups $(G_n)_{n\in\w}$. It turns out that this happens if and only if $G$ is the direct limit of this tower in the categories of topological groups or uniform spaces.

Let us recall that each topological group $G$ carries four natural uniformities:
\begin{itemize}
\item the {\em left uniformity} $\U^{\LL}$ generated by the entourages $U^{\LL}=\{(x,y)\in G^2:x\in yU\}$,
\item the {\em right uniformity} $\U^{\RR}$ generated by the entourages $U^{\RR}=\{(x,y)\in G^2:x\in Uy\}$,
\item the {\em two-sided uniformity} $\U^{\LR}$ generated by the entourages $U^{\LR}=\{(x,y)\in G^2:x\in yU\cap Uy\}$, and
\item the {\em Roelcke uniformity} $\U^{\RL}$ generated by the entourages $U^{\RL}=\{(x,y)\in G^2:x\in UyU\}$,
\end{itemize}
where $U=U^{-1}$ runs over open symmetric neighborhoods of the neutral element $e$ of $G$.

The group $G$ endowed with one of the uniformities $\U^\LL$, $\U^\RR$, $\U^{LR}$, $\U^{\RL}$ is denoted by $G^\LL$, $G^\RR$, $G^\LR$, $G^{\RL}$, respectively. These four uniformities on $G$ coincide if and only if the group $G$ is a SIN-{\em group}, which means that $G$ has a neighborhood base at $*_G$ consisting of open sets $U\subset G$ that are {\em invariant} in the sense that $U^G=U$ where $U^G=\{gug^{-1}:g\in G,\; u\in U\}$.

Let $G$ be a topological group and $(G_n)_{n\in\w}$ be a tower of closed subgroups of $G$ such that $G=\bigcup_{n\in\w}G_n$. Endowing the subgroups  $G_n$, $n\in\w$, with one of four canonical uniformities, we obtain four uniform direct limits $\ulim G_n^\LL$, $\ulim G_n^\RR$, $\ulim G_n^\LR$, and  $\ulim G_n^{\RL}$ of the towers of uniform spaces $(G_n^\LL)_{n\in\w}$, $(G_n^\RR)_{n\in\w}$, and $(G_n^\LR)_{n\in\w}$, $(G_n^\RL)_{n\in\w}$, respectively.

Besides those direct limits, the group $G$ also carries the topology of the {\em group direct limit} $\glim G_n$ of the tower $(G_n)_{n\in\w}$. This is the strongest topology that turns $G=\bigcup_{n\in\w}G_n$ into a topological group and makes the identity maps $G_n\to G$, $n\in\w$, continuous.

For these direct limits we get the following diagram in which each arrow indicates that the corresponding identity map is continuous:
$$
\xymatrix{
&&\ulim G_n^\LL\ar[rd]&\\
\tlim G_n\ar[r]&\ulim G^{\LR}_n\ar[ru]\ar[rd]& &\ulim G^{\RL}_n\ar[r]&\glim G_n\ar[r]&G\\
&&\ulim G_n^\RR\ar[ru]
}
$$

The following proposition was proved in \cite{BRTohoku}.

\begin{proposition}\label{u=g} For a topological group $G$ and a tower of subgroups $(G_n)_{n\in\w}$ with $G=\bigcup_{n\in\w}G_n$ the following conditions are equivalent:
\begin{enumerate}
\item $G$ carries the strong topology with respect to the tower $(G_n)_{n\in\w}$;
\item the identity map $\ulim G_n^{\LL}\to G$ is a homeomorphism;
\item the identity map $\ulim G_n^{\RR}\to G$ is a homeomorphism.
\end{enumerate}
The equivalent conditions (1)--(3) imply:
\begin{enumerate}
\item[(4)] the identity map $\glim G_n\to G$ is a homeomorphism.
\end{enumerate}
\end{proposition}

It should be mentioned that for a tower of metrizable topological groups $(G_n)_{n\in\w}$ the identity map $\tlim G_n\to\ulim G_n$ is a homeomorphism if and only if all groups $G_n$ are locally compact or there is a number $m\in\w$ such that for every $n\ge m$ the group $G_n$ is open in $G_{n+1}$, see \cite{BRTohoku}, \cite{BaZ}, \cite{Yama} or \cite[7.1]{Glock}.

If $(X_n)_{n\in\w}$ is a tower of locally convex linear topological spaces, then besides the topology of the group direct limit $\glim X_n$, the union $X=\bigcup_{n\in\w}X_n$ carries also the topology of the direct limit
in the category of (locally convex) linear topological spaces. The corresponding direct limit space is denoted by $\lclim X_n$ (resp. $\llim X_n$). This is the union $X=\bigcup_{n\in\w}X_n$ endowed with the strongest topology that turns $X$ into a (locally convex) linear topological space and makes the identity maps $X_n\to X$, $n\in\w$, continuous.

The following proposition proved in \cite{BR10} implies that many direct limit topologies on $X$ coincide.

\begin{proposition} For any tower $(X_n)_{n\in\w}$ of locally convex linear topological spaces the identity maps
$$\ulim X_n\to \glim X_n\to\llim X_n\to\lclim X_n$$
are homeomorphisms.
\end{proposition}

In particular, each LF-space has the topology of the uniform direct limit of a tower of Fr\'echet spaces.

\section{(Locally) topologically complemented subgroups in topological groups}

In this section we study (locally) topologically complemented subgroups of topological groups. Let us recall that a closed subgroup $H$ of a topological group is ({\em locally}) {\em topologically complemented} if the quotient map $q:G\to G/H$ is a (locally) trivial bundle.
Here  $G/H=\{xH:x\in G\}$ is the quotient space of left cosets of $H$ in $G$. It is a pointed topological space with a distinguished point $*_{G/H}=H$.
For the theory of bundles, we refer the reader to \cite{Hus}.

\begin{lemma}\label{l:cmp} If $H$ is a (locally) topologically complemented subgroup of a topological group $G$, then $H$ is (locally) $G/H$-complemented in the uniform space $G^\RR$.
\end{lemma}

\begin{proof} Since $H$ is locally topologically complemented in $G$, the quotient map $q:G\to G/H$ has a continuous section $s:U\to G$ defined in an open neighborhood $U\subset G/H$ of the distinguished point $*_{G/H}$.
If $H$ is topologically complemented in $G$, then we can take $U$ to be equal to $G/H$. Replacing $s(x)$ by $s(x)s(*_{G/H})^{-1}$, we can additionally assume that $s(*_{G/H})$ coincides with the neutral element $*_G$ of the group $G$. It follows from the definition of the uniformity $\U^\RR$ that the preimage $q^{-1}(U)$ is an open uniform neighborhood of $H$ in the uniform space $G^\RR$. Now we see that the homeomorphism
$$\gamma:H\times U\to q^{-1}(U),\;\;\gamma(h,y)=s(y)\cdot h$$witnesses that $H$ is $U$-complemented in the uniform neighborhood $q^{-1}(U)$ of $H$ in $G^\RR$ and hence $H$ is locally $G/H$-complemented in $G^\RR$.

If $U=G/H$, then the homeomorphism $\gamma$ witnesses that $H$ is $G/H$-complemented in $G^\RR$.
\end{proof}

In some cases the local topological complementability implies the topological complementability.

\begin{proposition}\label{p:ltc-tc} A locally topologically complemented subgroup $H$ of a topological ANR-group $G$ is topologically complemented if the quotient space $G/H$ is contractible or both spaces $G$ and $H$ are contractible.
\end{proposition}

\begin{proof} First we show that the quotient space $G/H$ is a (metrizable) ANR. Being metrizable, the group $G$ admits a right invariant metric $d$
generating the topology of $G$.
Then the topology of the quotient space $G/H$ is generated by
the metric $\rho$ defined by
$$\rho(xH,yH)=\inf\{d(a,b):a\in xH,\; b\in yH\},\;\; xH,yH\in G/H.$$
Hence $G/H$, being metrizable, is paracompact.

Since $H$ is locally topologically complemented in $G$, the quotient map $q:G\to G/H$ is a locally trivial bundle with fiber $H$. This implies that the space $H\times (G/H)$ is locally homeomorphic to $G$. Since $G$ is an ANR, each point of the quotient space $G/H$ has an ANR-neighborhood,
which implies that $G/H$ is an ANR, see Hanner's Theorem 5.1 in \cite[Ch.II]{BP}.

If the quotient space $G/H$ is contractible, then the locally trivial
bundle $q:G\to G/H$ is trivial according to  Corollary 4.10.3 of \cite{Hus}.

Now assume that the spaces $G$ and $H$ are contractible. We claim that the quotient space $G/H$ contractible.
Since $G$ is path connected, so is $G/H$.
Since both the total space $G$ and the fiber $H$ of the bundle $q:G\to G/H$ are contractible,
the exact sequence of the fibration $\pi : G \to G/H$ implies
that all homotopy groups $\pi_i(G/H) = 0$, $i\in\IN$, of $G/H$ are trivial.
Therefore, by Whitehead Theorem (see II.6.1 in \cite{BP}),
the ANR-space $G/H$ is contractible.
\end{proof}

\section{Proof of Theorem~\ref{t:gbox}}\label{s:gbox}

In this section we shall prove Theorem~\ref{t:gbox}. Assume that a topological group $G$ carries the strong topology with respect to a (locally) topologically complemented tower $(G_n)_{n\in\w}$ of subgroups.

By Proposition~\ref{u=g}, the topology of $G$ coincides with the topology of the uniform direct limit $\ulim G_n^\RR$ of the tower $(G_n^\RR)_{n\in\w}$ of the groups $G_n$ endowed with their right uniformities. By Lemma~\ref{l:cmp}, each set $G_n$ is (locally) $G_{n+1}/G_n$-complemented in $G_{n+1}^\RR$. Taking into account that the space $G$ is topologically homogeneous and applying Lemma~\ref{l:udlbp}, we conclude that $G=\ulim G_n^\RR$ is (locally) homeomorphic to the small box-product $G_0\times\cbox_{n\in\w}G_{n+1}/G_n$.

\section{The $Z$-point property in quotient spaces of topological groups}

In this section we shall study the $Z$-point property in quotient spaces of topological groups. Let us recall that a topological space $X$ has {\em the $Z$-point property} if each $Z$-point in $X$ is a strong $Z$-point.
In fact, it is more convenient to work not with (strong) $Z$-point but with an equivalent notion of a (strong) $Z_\infty$-point.

Let $\kappa$ be a cardinal. A closed subset $A$ of a topological space $X$ is called a {\em $(\kappa{\times}Z_\infty)$-set} if for each open cover $\U$ of $X$ any map $f:\kappa\times\II^\w\to X$ can be approximated by a map $\tilde f:\kappa\times \II^\w\to X$ such that $\tilde f$ is $\U$-near to $f$ and $A$ does not intersect the closure of the set $\tilde f(\kappa\times \II^\w)$ in $X$.

We shall refer to $(1\times Z_\infty)$-sets and $(\w\times Z_\infty)$-sets as {\em $Z_\infty$-sets} and {\em strong $Z_\infty$-sets}, respectively. Such sets were studied in \cite{Tor78}, \cite[\S 2.2]{Chi} and \cite[\S1.4]{BRZ}. A point $x$ of a topological space $X$ will be called a ({\em strong}) {\em $Z_\infty$-point} if the singleton $\{x\}$ is a (strong) $Z_\infty$-set in $X$.

The following characterization of (strong) $Z$-sets in (separable) ANR's is well-known, see \cite{Tor78} and \cite[2.2.3 and 2.2.6]{Chi}.

\begin{lemma}\label{l:ZsZ} A point $x$ of a (separable) ANR-space $X$ is a (strong) $Z$-point if and only if it is a (strong) $Z_\infty$-point.
\end{lemma}

\begin{lemma}\label{l:Zinfty} Let $H$ be a locally topologically complemented subgroup in a topological group $G$. If the quotient space $G/H$ is not locally compact, then each compact subset $K$ of $G/H$ is a $Z_\infty$-set in $G/H$.
\end{lemma}

\begin{proof} Assume that the quotient space $G/H$ is not locally compact and take any compact subset $K\subset G/H$. To prove that $K$ is a $Z_\infty$-set in $G/H$, fix an open cover $\U$ of $G/H$ and a continuous map $f:\II^\w\to G/H$ from  the Hilbert cube.

By the local complementability of the subgroup $H$ in $G$, the quotient map $q:G\to G/H$ is a locally trivial bundle. Using this fact and the contractibility of the Hilbert cube $\II^\w$, we can find a continuous map $g:\II^\w\to G$ such that $q\circ g=f$. By compactness of $g(\II^\w)$, there is a neighborhood $U\subset G$ of the neutral element $*_G$ of $G$ so small that for every $u\in U$ the map $f_u:\II^\w\to G/H$ defined by $f_u(x)=q(g(x)u)$ for $x\in\II^\w$ is $\U$-near to $f$. Since the quotient map $q:G\to G/H$ is open, the set $q(U)$ is an open neighborhood of $*_{G/H}$.

By the local triviality of $q$, there is a compact subset $\tilde K\subset G$ such that $q(\tilde K)=K$. Consider the compact subset $C=g(\II^\w)$ of $G$ and the compact subset $q(C^{-1}\tilde K)\subset G/H$, which does not contain the neighborhood $q(U)$ of $*_{G/H}$ as $G/H$ is not locally compact at $*_{G/H}$. Consequently, there is an element $u\in U$ with $q(u)\notin q(C^{-1}\tilde K)$, which implies that $u\notin C^{-1}\tilde KH$
and hence $Cu\cap \tilde KH=\emptyset$. Then the map $f_u:\II^\w\to G/H$ has the required property: it is $\U$-near to $f$ and $f_u(\II^\w)\cap K=q(g(\II^\w)\cdot u)\cap q(\tilde KH)=\emptyset$.
\end{proof}

\begin{proposition}\label{p:lz} Let $H$ be a locally topologically complemented subgroup of an ANR-group $G$. The quotient space $G/H$ is a $\lz$-pointed space if and only if it is not discrete and has the $Z$-point property.
\end{proposition}

\begin{proof} To prove the ``only if'' part, assume that the quotient space $G/H$ is an $\lz$-space. Then the distinguished point $*_{G/H}$ of $G/H$ is not isolated and hence $G/H$ is not discrete. If $G/H$ is locally compact, then by \cite[2.2.4]{Chi}, each $Z$-set in $G/H$ is a strong $Z$-set and hence $G/H$ has the $Z$-point property. If $G/H$ is not locally compact, then $*_{G/H}$ is a strong $Z$-point in $G/H$ and by the topological homogeneity of $G/H$, each point of $G/H$ is a strong $Z$-point. Then the space $G/H$ trivially has the $Z$-point property.

To prove the ``if'' part, assume that the quotient space $G/H$ is not discrete and has the $Z$-point property. We should prove that the distinguished point $*_{G/H}$ of $G/H$ is not isolated and either $G/H$ is locally compact or $*_{G/H}$ is a strong $Z$-point in $G/H$.

The space $G/H$ is not discrete and hence contains a non-isolated point. Then by the topological homogeneity of $G/H$, no point of $G/H$ is isolated. If $G/H$ is not locally compact, then $*_{G/H}$ is a $Z_\infty$-point according to Lemma~\ref{l:Zinfty}.
Since $G/H$ is an ANR as shown in the proof of Proposition~\ref{p:ltc-tc},
by Lemma~\ref{l:ZsZ}, $*_{G/H}$ is a $Z$-point in $G/H$ and by the $Z$-point property, it is a strong $Z$-point in $G/H$.
\end{proof}

Finally, we consider the problem of detecting quotient spaces $G/H$ all whose points are  strong $Z$-points. Let us recall that an action of a topological group $G$ on a topological space $X$ is called {\em locally bounded} at a point $x_0\in X$ if there is a neighborhood $U\subset G$ of the neutral element $*_G$ of $G$ such that for every neighborhood $V\subset X$ of $x_0$ there is a compact subset $K\subset X$ that meets each shift $xV$, $x\in U$. In the opposite case the action is called {\em locally unbounded}.

\begin{proposition}\label{p:Zlb} Let $H$ be a closed subgroup of a locally path-connected group $G$. If the action of $G$ on the quotient space $G/H$ is locally unbounded, then $*_{G/H}$ is a strong $Z_\infty$-point in $G/H$.
If $G/H$ is a separable ANR-space, then $*_{G/H}$ is a strong $Z$-point in $G/H$.
\end{proposition}

\begin{proof} Fix an open cover $\U$ of $G/H$ and a continuous map $f:\w\times\II^\w\to G/H$. Find a set $U_0\in\U$ that contains the distinguished point $*_{G/H}$. The continuity of the action of $G$ on $G/H$ yields a neighborhood $V_1\subset G$ of $*_G$ and a neighborhood $U_1\subset G/H$ of $*_{G/H}$ such that $V_1\cdot U_1\subset U_0$. Since the quotient space $G/H$ is completely regular, there is a continuous function $\lambda:G/H\to [0,1]$ such that $G/H\setminus U_1\subset \lambda^{-1}(0)$ and $\lambda^{-1}(1)$ is a neighborhood of $*_{G/H}$. By continuity of the action of $G$, there are neighborhoods $V_2\subset G$ and $U_2\subset G/H$ of $*_G$ and $*_{G/H}$ such that $V_2^{-1}\subset V_1$ and $V_2\cdot U_2\subset \lambda^{-1}(1)$. Since $G$ is locally path-connected, there is a neighborhood $V_3\subset G$ of $*_G$ such that each point $x\in V_3$ can be connected with $*_G$ by a continuous path lying in $V_2$.

Since the action of $G$ in $G/H$ is locally unbounded, for the neighborhood $V_3$ there is a neighborhood $U_3\subset U_2$ of $*_{G/H}$ such that for any compact subset $K\subset G/H$ there is a point $x\in V_3$ such that the shift $xU_3$ does not intersect $K$.

Now we shall construct a map $\tilde f:\w\times \II^\w\to G/H$ such that
$\tilde f$ is $\U$-near to $f$ and $\tilde f(\w\times\II^\w)\cap U_3=\emptyset$.
It suffices for every $n\in\w$ to approximate the map $f_n:\II^\w\to G/H$, $f_n:t\mapsto f(n,t)$, by a map $\tilde f_n:\II^\w\to G/H$ such that $\tilde f_n$ is $\U$-near to $f_n$ and $\tilde f_n(\II^\w)\cap U_3=\emptyset$.

For every $n\in\w$ consider the compact subset $K_n=f(\{n\}\times\II^\w)$ of $G/H$. By the choice of $U_3$ there is a point $x_n\in V_3$ such that $x_nU_3\cap K_n=\emptyset$. Then $x_n^{-1}K_n\cap U_3=\emptyset$. By the choice of the neighborhood $V_3$ there is a continuous path $\gamma_n:[0,1]\to V_2$ such that $\gamma_n(0)=*_G$ and $\gamma_n(1)=x_n$. Define the map $\tilde f_n:\II^\w\to G/H$ by the formula $$\tilde f_n(t)=\gamma_n(\lambda(f_n(t)))^{-1}f_n(t)$$ for $t\in\II^\w$.

\begin{claim}\label{cl1} The map $\tilde f_n$ is $\U$-near to $f_n$.
\end{claim}

\begin{proof} Fix any $t\in\II^\w$. If $\lambda(f_n(t))=0$, then $\tilde f_n(t)=*_G\cdot f_n(t)$ and hence $\{\tilde f_n(t),f_n(t)\}$ is a singleton lying in some element of the cover $\U$.

If $\lambda(f_n(t))>0$, then $f_n(t) \in U_1$ and $\tilde f_n(t)=\gamma_n(\lambda(f_n(t)))^{-1}\cdot f_n(t)\subset V_2^{-1}\cdot U_1\subset V_1\cdot U_1\subset U_0$. In this case $\{f_n(t),\tilde f_n(t)\}\subset U_0\in\U$.
\end{proof}

\begin{claim}\label{cl2} $\tilde f_n(\II^\w)\cap U_3=\emptyset$.
\end{claim}

\begin{proof} Given any point $t\in\II^\w$, we must prove that $\tilde f_n(t)\notin U_3$.

If $\lambda(f_n(t))<1$, then $f_n(t)\notin V_2\cdot U_2$. We claim that $\tilde f_n(t)\notin U_2$. Otherwise $$
f_n(t)=\gamma_n(\lambda(f_n(t)))\tilde f_n(t)\in V_2 \cdot U_2.$$
Then $\tilde f_n(t)\in (G/H)\setminus U_2\subset (G/H)\setminus U_3$.

If $\lambda(f_n(t))=1$, then
$$\tilde f_n(t)=\gamma_n(\lambda(f_n(t)))^{-1}f_n(t)=x_n^{-1}f_n(t)\in x_n^{-1}K_n\subset (G/H)\setminus U_3$$by the choice of the point $x_n$.
\end{proof}

Claims~\ref{cl1} and \ref{cl2} complete the proof that $*_{G/H}$ is a strong $Z_\infty$-point. If $G/H$ is a separable ANR, then $*_{G/H}$ is a strong $Z$-point, according to Lemma~\ref{l:ZsZ}.
\end{proof}

\section{Topological groups, (locally) homeomorphic to small box-products of typical model spaces}

In fact, Theorem~\ref{t:gLF} holds in a more general context of small box-products of typical model spaces.
Typical model spaces were introduced in \cite{BR-LF} so that manifolds modeled on such spaces have many common properties with Hilbert manifolds.

\begin{definition}\label{d:tms} A pointed topological space $E$ is called a {\em typical model space} if
\begin{enumerate}
\item $E$ is a topologically homogeneous absolute retract containing a topological copy of the Hilbert cube $Q=[0,1]^\w$;
\item For any neighborhood $U\subset E$ of $*_E$ there is a neighborhoods $V,W\subset U$ of $*_E$ such that $W$ and $E\setminus V$ are homeomorphic to $E$ and the boundary $\partial V$ of $V$ is a retract of $\overline{V}$ and a $Z$-set in $E\setminus V$;
\item each contractible $E$-manifold is homeomorphic to $E$;
\item each connected $E$-manifold $M$ is homeomorphic to an open subset of $E$;
\item any homeomorphism $h:A\to B$ between $Z$-sets $A,B\subset E$ extends to a homeomorphism $\bar h:E\to E$ of $E$;
\item for any $E$-manifold $M$ the projection $E\times M\to M$ is a near homeomorphism;
\item for any retract $X$ of an open subset of $E$ the product $X\times E$ is homeomorphic to an open subset of $E$;
\item for any retract $X$ of an $E$-manifold and a strong $Z$-point $*_X\in X$ the reduced product $X\rtimes E$ is an $E$-manifold, homeomorphic to   $X\times E$.
\end{enumerate}
\end{definition}

\begin{remark} A typical example of a completely metrizable typical model space is any infinite-dimensional Hilbert space, see \cite[4.2]{BR-LF}. Many incomplete typical model spaces can be found among absorbing and coabsorbing spaces, see \cite{BM} and \cite{BRZ}.
\end{remark}

The following criterion was proved in \cite[6.1]{BR-LF}.

\begin{theorem}\label{t:tsbox} Let $(X_n)_{n\in\w}$ be a sequence of pointed topological spaces such that for every $n\in\w$ the finite product $\prod_{i\le n}X_n$ is homeomorphic to (an open subspace of) some typical model space $E_n$. Assume that for infinitely many numbers $n\in\w$ the pointed space $X_n$ is $\lz$-pointed. Then the small box-product $\cbox_{n\in\w}X_n$ is homeomorphic to (an open subset of) the small box-product $\cbox_{n\in\w}E_n$.
\end{theorem}

Theorems~\ref{t:gbox}, \ref{t:tsbox} and Propositions~\ref{p:lz} and
\ref{p:ltc-tc} imply the following ``typical'' version of Theorem~\ref{t:gLF}.

\begin{theorem} Let $(E_n)_{n\in\w}$ be a sequence of typical model spaces. A topological group $G$ carrying the strong topology with respect to a tower of subgroups $(G_n)_{n\in\w}$ is
\begin{enumerate}
\item homeomorphic to (an open subset of) $\cbox_{n\in\w}E_n$ if for every $n\in\w$ the group $G_n$ is homeomorphic (to an open subset of) $E_n$, $G_n$ is topologically complemented in $G_{n+1}$ and for infinitely many numbers $n\in\w$ the quotient space $G_{n+1}/G_n$ is not discrete and has the $Z$-point property;
\item (locally) homeomorphic to $\cbox_{n\in\w}E_n$ if for every $n\in\IN$ the  group $G_n$ is (locally) homeomorphic to $E_n$ and is locally topologically complemented in $G_{n+1}$ and for infinitely many numbers $n\in\w$ the quotient space $G_{n+1}/G_n$ is not discrete and has the $Z$-point property.
\end{enumerate}
\end{theorem}

\end{document}